\documentclass[11pt]{article}
\newcommand{\authorfootnotes}{\renewcommand\thefootnote{\@fnsymbol\c@footnote}}%
\usepackage{amssymb,amsmath, amsfonts, amsthm}

\usepackage{graphicx}
\usepackage{pgfplots}
\usepackage{tikz}
\usepackage{caption}
\usepackage{booktabs}
\usepackage{array}
\usepackage{verbatim}
\usepackage{subfig}
\usepackage{geometry}
\usepackage{fancyhdr}

\newcommand{\cp}{\mathbin{\Box}}
\newcommand{\vertex}{\node[vertex]}
\tikzstyle{vertex}=[circle, draw, inner sep=0pt, minimum size=6pt]
\usetikzlibrary{arrows}

\newtheorem{thm}{Theorem}
\newtheorem{lem}{Lemma}
\newtheorem{cor}{Corollary}
\newtheorem{prop}{Proposition}
\newtheorem{prob}{Problem}
\newtheorem{conj}{Conjecture}
\newtheorem{claim}{Claim}

\newtheorem{obs}{Observation}

\newcommand{\ideg}{{\rm indeg}}
\newcommand{\odeg}{{\rm outdeg}}
\newcommand{\opack}{\rho^{\rm o}}

\newcommand{\cing}{{\mathcal N}^-_c}
\newcommand{\oing}{{\mathcal N}^-_o}
\newcommand{\cng}{{\mathcal N}_c}

\newcommand{\ugr}[1]{\widehat{#1}}

\begin{document}

\title{Domination in digraphs and their products}

\author{
Bo\v{s}tjan Bre\v{s}ar$^{a,b}$
\and
Kirsti Kuenzel$^{c}$
\and
Douglas F. Rall$^{d}$\\
}

%\date{\today}

\maketitle

\begin{center}
$^a$ Faculty of Natural Sciences and Mathematics, University of Maribor, Slovenia\\

$^b$ Institute of Mathematics, Physics and Mechanics, Ljubljana, Slovenia\\
$^c$ Department of Mathematics, Trinity College, Hartford, CT, USA\\
$^d$ Department of Mathematics, Furman University, Greenville, SC, USA\\
\end{center}

\begin{abstract}
A  dominating (respectively, total dominating) set $S$ of a digraph $D$ is a set of vertices in $D$ such that the union of the closed (respectively, open) out-neighborhoods of vertices in $S$ equals the vertex set of $D$. The minimum size of a dominating (respectively, total dominating) set of $D$ is the  domination (respectively, total domination) number of $D$, denoted $\gamma(D)$ (respectively,~$\gamma_t(D)$). The maximum number of pairwise disjoint closed (respectively,~open) in-neighborhoods of $D$ is denoted by $\rho(D)$ (respectively,~$\opack(D)$).
We prove that in digraphs whose underlying graphs have girth at least $7$, the closed (respectively,~open) in-neighborhoods enjoy the Helly property, and use these two results to prove that in any ditree $T$ (that is, a digraph whose underlying graph is a tree), $\gamma_t(T)=\rho^o(T)$ and $\gamma(T)=\rho(T)$. By using the former equality we then prove that $\gamma_t(G\times T)=\gamma_t(G)\gamma_t(T)$, where $G$ is any digraph and $T$ is any ditree, each without a source vertex, and $G\times T$ is their direct product. From the equality $\gamma(T)=\rho(T)$ we derive the bound $\gamma(G\cp T)\ge\gamma(G)\gamma(T)$, where $G$ is an arbitrary digraph, $T$ an arbitrary ditree and $G\cp T$ is their Cartesian product. In general digraphs this Vizing-type bound fails, yet we prove that for any digraphs $G$ and $H$, where $\gamma(G)\ge\gamma(H)$, we have $\gamma(G \cp H) \ge \frac{1}{2}\gamma(G)(\gamma(H) + 1)$. This inequality is sharp as demonstrated by an infinite family of examples. Ditrees $T$ and digraphs $H$ enjoying $\gamma(T\cp H)=\gamma(T)\gamma(H)$ are also investigated.
\end{abstract}


\noindent
{\bf AMS Subj.\ Class.\ (2010)}: 05C20, 05C69, 05C76.

\maketitle

\section{Introduction}

  Investigations of domination in digraphs have had a different focus from those in graphs. A large part of the studies of domination in digraphs was concentrated around the concept of the kernel, cf.~\cite[Section 7.7]{hhs1} and~\cite[Chapter 15]{hhs2}. Since a digraph $D=(V(D),A(D))$ can be considered as a generalization of a graph (the set of arcs $A(D)$ in a graph is restricted to be a symmetric relation on the set of vertices $V(D)$), it is natural to ask, which of the important results on domination in graphs extend to the context of digraphs. In this paper, we answer several instances of such questions.

A classical result of Meir and Moon from 1975 states that the domination number of a tree $T$ equals the maximum number of pairwise disjoint closed neighborhoods in $T$ (called the {\em $2$-packing number} of $T$ and denoted $\rho_2(T)$):
\begin{thm} \label{thm:mm}{\rm (\cite{mm-1975})}
If $T$ is a tree, then $\rho_2(T)=\gamma(T)$.
\end{thm}
\noindent In Section~\ref{sec:p-d-graphs}, we give a short alternative proof of this result by combining the Helly property of the closed neighborhoods in trees with the Perfect Graph Theorem of Lov\' asz~\cite{lo-1972}. The idea of this alternative proof is then transferred to digraphs, and in Section~\ref{sec:p-d-digraphs} we prove the digraph version of Theorem~\ref{thm:mm}:
\begin{thm}  \label{thm:packing-dom-ditrees}
For any ditree $T$, $\rho(T) = \gamma(T)$.
\end{thm}
\noindent  Theorem~\ref{thm:packing-dom-ditrees} extends the recent result of Mojdeh, Samadi and Gonz\'{a}lez Yero~\cite[Theorem 5]{msy-2019} who proved it for oriented trees.  (An oriented tree is a digraph obtained from an undirected tree by orienting each of its edges in one direction.)

The following analogue of Theorem~\ref{thm:packing-dom-ditrees} for the total domination number in ditrees, extends a result of Rall~\cite{ra-2005} from the graph context to digraphs.

\begin{thm}  \label{thm:openpacking-tdom-ditrees}
If $T$ is a ditree such that $\delta^-(T) \ge 1$, then $\opack(T) = \gamma_t(T)$.
\end{thm}
\noindent We then consider the direct products of digraphs, and prove the following equality for their total domination numbers (again, this result is a generalization of an analogous result on graphs by Rall~\cite{ra-2005}):

\begin{thm} \label{thm:equaltotalfordirect}
If $G$ is a digraph with $\delta^-(G) \ge 1$ such that $\opack(G)=\gamma_t(G)$, then $\gamma_t(G \times H)=\gamma_t(G)\gamma_t(H)$, for every digraph $H$ with minimum in-degree at least $1$.
\end{thm}

\noindent {The above theorems provide the equality $\gamma_t(G \times H)=\gamma_t(G)\gamma_t(H)$, whenever $G$ is a ditree with $\delta^-(G) \ge 1$.

The following conjecture concerning the domination number of a Cartesian product of (undirected) graphs was made by Vizing in~1968.

\begin{conj}\label{VC}{\rm (\cite{v1968})}
If $G$ and $H$ are graphs, then
\begin{equation}\label{VCeqn}
  \gamma(G\cp H) \ge \gamma(G)\,\gamma(H).
\end{equation}
\end{conj}
\noindent In spite of numerous attempts to resolve the conjecture, it is still open; see the survey~\cite{BDGHHKR2012} containing many partial results on the conjecture and other developments.

For convenience sake we shall refer to the inequality in~\eqref{VCeqn} as \emph{Vizing's inequality}.  A digraph $G$ {\em satisfies Vizing's inequality} if
$\gamma(G\cp H)\ge\gamma(G)\gamma(H)$ for every digraph $H$.
In Section~\ref{sec:results}, we present infinite families of digraphs that do not satisfy Vizing's inequality.  However,
we prove the following (weaker) lower bound on the domination number of the Cartesian product of two digraphs:
\[
\gamma(G \cp H) \ge \frac{1}{2}\gamma(G)\gamma(H) + \frac{1}{2}\max\{\gamma(G),\gamma(H)\}\,,
\]
and present an infinite family of digraphs that attain the bound.

 One of the most basic lower bounds, which approximates the lower bound in Vizing's inequality, is obtained when one of the numbers $\gamma(G)$ or $\gamma(H)$ is replaced by the packing number of the corresponding factor graph. This result extends easily to digraphs, as we will show in Section~\ref{sec:results}.

\begin{prop} \label{prp:packing}
If $G$ and $H$ are digraphs, then
\[
\gamma(G \cp H) \ge \max\{\gamma(G)\rho(H),\gamma(H)\rho(G)\}\,.
\]
\end{prop}
\noindent Theorem~\ref{thm:packing-dom-ditrees} combined with Proposition~\ref{prp:packing} then implies that every ditree $T$ satisfies Vizing's inequality. In Section~\ref{sec:equality}, we investigate pairs of digraphs $G$ and $H$ for which the equality $\gamma(G \cp H) =\gamma(G)\gamma(H)$ is satisfied. Some partial results on such ditrees are obtained, and we present a large class of trees that satisfy the equality when multiplied with $P_4$. A number of open problems are posed along the way.

\section{Definitions and notation} \label{sec:defnot}

For undirected graphs we mostly follow the definitions and notation found in~\cite{imklra-2008}.
A \emph{directed graph} $G$, or \emph{digraph} for short, consists of a set of vertices $V(G)$ along with a binary relation $A(G)$ on $V(G)$.  We call an element
$(u,v)$ of $A(G)$ an \emph{arc}.  If $(u,v) \in A(G)$, then we say that $u$ \emph{is adjacent to} $v$ or that $v$ \emph{is adjacent from} $u$.  If we simply want to indicate
that at least one of the arcs $(u,v)$ and $(v,u)$ is in the digraph $G$, then we will say that $u$ and $v$ are \emph{adjacent} in $G$.
In this paper we will be assuming that the binary relation is irreflexive (that is, we are assuming that $(x,x) \notin A(G)$ for every $x \in V(G)$).  Also, when there is no chance of confusion, we will shorten the notation for an arc from $(u,v)$ to $uv$.

A vertex $u$ is an \emph{in-neighbor} of $v$ if $uv \in A(G)$ and an \emph{out-neighbor} of $v$ if $vu \in A(G)$. The \emph{open out-neighborhood} of $v$ is the set of out-neighbors of $v$ and is denoted by $N^+_G(v)$. The \emph{closed out-neighborhood} of $v$ is the set $N^+_G[v]$ defined by $N^+_G[v]=N^+_G(v)\cup \{v\}$.  In a similar manner, the \emph{open in-neighborhood} of $v$ is the set $N^-_G(v)$, which consists of all the in-neighbors of $v$; the \emph{closed in-neighborhood} of $v$ is the set $N^-_G[v]=N^-_G(v)\cup \{v\}$.   The \emph{in-degree} of $v$ is the number $\ideg(v)=|N^-_G(v)|$ and the \emph{out-degree} of $v$ is $\odeg(v)=|N^+_G(v)|$.  The minimum of  the in-degrees (respectively, out-degrees) among all the vertices of $G$ is denoted $\delta^-(G)$ (respectively, $\delta^+(G)$).  If $S \subseteq V(G)$, then
\[N^-_G(S)=\bigcup_{v\in S}N^-_G(v),  \hskip 2cm N^-_G[S]=\bigcup_{v\in S}N^-_G[v]\,,\]
\[ N^+_G(S)=\bigcup_{v\in S}N^+_G(v), \hskip .8cm \text{and} \hskip .8cm  N^+_G[S]=\bigcup_{v\in S}N^+_G[v]\,.\]
If the digraph $G$ is clear from the context, then we will often omit the subscript $G$ from the above notations.

In certain instances we will be considering both digraphs and graphs, in which case we may say that the graph is an \emph{undirected graph} for emphasis.
For a given digraph $G$, the \emph{underlying graph} of $G$ is the undirected graph, denoted $\ugr{G}$, with vertex set $V(G)$ in which $\{u,v\}$ is an edge in $\ugr{G}$ precisely when one or both of the arcs $uv$ and $vu$ belongs to the digraph $G$.  If the underlying graph is a tree, then we call the digraph a \emph{ditree}.    If $v$ is a vertex in $G$, then we call $v$ a \emph{leaf} of $G$ if it is a leaf in the underlying graph $\ugr{G}$. If $v$ is a leaf with no in-neighbors, we call it an \emph{isolated leaf}.  Otherwise, we refer to $v$ as a \emph{non-isolated leaf}. We call $u$ a \emph{support vertex} of $G$ if $u$ is a support vertex in $\ugr{G}$. Similarly, we call $u$ a \emph{strong support vertex} of $G$ if $u$ is a strong support vertex (that is, $u$ has at least two neighbors that are leaves) in $\ugr{G}$.  If $u$ is a support vertex, then any leaf adjacent to or adjacent from $u$ will be called a \emph{leaf of $u$}. An undirected graph $C$ is a \emph{corona} if there exists an undirected graph $H$ and $C$ is obtained from $H$ by the following construction.  For each $h \in V(H)$ add a new vertex $h'$ and a new edge $hh'$.

Given an undirected graph $H$, an \emph{orientation of $H$} is the digraph obtained from $H$ by replacing each edge $\{x,y\}$ of $H$ by exactly one of the arcs $(x,y)$ or $(y,x)$.  In this case the resulting digraph is called an \emph{oriented graph}.  When the undirected graph is a tree or a cycle we called the resulting digraph an \emph{oriented tree} or an \emph{oriented cycle}, respectively.

Let $D$ be a digraph. A vertex $v$ of $D$ is said to \emph{dominate} itself and each of its out-neighbors.  We say that a subset $S\subseteq V(D)$ is a \emph{dominating set} of $D$ if for every $x\in V(D)-S$ there exists $v\in S$ such that  $(v,x)$ is an arc in $D$. Equivalently, $S$ is a dominating set of $D$ if $V(D)=N^+[S]$. The minimum size of a dominating set of $D$ is the \emph{domination number} of $D$ and is denoted $\gamma(D)$.  Any dominating set of $D$ that has cardinality $\gamma(D)$ will be called a $\gamma(D)$-set. A vertex \emph{totally dominates} each of its out-neighbors.  A subset $S \subseteq V(D)$ is a \emph{total dominating set} of $D$ if each vertex $x \in V(D)$ is totally dominated by at least one vertex in $S$; that is, if $N_D^+(S)=V(D)$.  Note that a digraph admits a total dominating set only if its minimum in-degree is at least $1$.  If $\delta^-(D)\geq 1$, then the \emph{total domination number} of $D$ is the minimum cardinality of a total dominating set of $D$; the total domination number of $D$ is denoted $\gamma_t(D)$.

A subset $P$ of $V(D)$ is a \emph{packing} of $D$ if there are no arcs joining vertices of $P$ and for every two vertices $x,y\in P$ there does not exist $v\in V(D)$
such that $\{(v,x),(v,y)\} \subseteq A(D)$.  In other words, by definition no two vertices of a packing are dominated by a single vertex.  An equivalent way to specify this is that $P$ is a packing if and only if $N^-[x]\cap N^-[y]=\emptyset$, for every pair of distinct vertices $x,y \in P$.  The cardinality of a largest packing in $D$ is denoted $\rho(D)$, and is called the \emph{packing number} of $D$.  Thus, $\rho(D)$ is the cardinality of the largest subset $P$ of $V(D)$ such that the closed in-neighborhoods of vertices in $P$ are pairwise disjoint.  A subset $B$ of $V(D)$ is an \emph{open packing} of $D$ if $N_D^-(u) \cap N_D^-(v) = \emptyset$, for every pair of distinct vertices
$u$ and $v$ in $B$.  The \emph{open packing number}, $\opack(D)$, of $D$ is the maximum cardinality of an open packing in $D$.  For an undirected graph $G$ a subset $A \subseteq V(G)$ is a \emph{$2$-packing} in $G$ if the closed neighborhoods of the vertices in $A$ are pairwise disjoint.  The \emph{$2$-packing number} of $G$ is the cardinality, $\rho_2(G)$, of a largest $2$-packing in $G$. The \emph{open packing number} of an undirected graph $G$ is the maximum cardinality of a subset of vertices in $G$ no two of which share a common neighbor; it is denoted $\opack(G)$.

The \emph{Cartesian product} $G\cp H$ of digraphs $G$ and $H$ is the digraph whose vertex set is $V(G) \times V(H)$, and  $(g_1,h_1)$ is adjacent to  $(g_2,h_2)$ in $G \cp H$ if either $g_1=g_2$ and $(h_1,h_2)$ is an arc in $H$, or $h_1=h_2$ and $(g_1,g_2)$ is an arc in $G$. The digraphs $G$ and $H$ are called the \emph{factors} of the Cartesian product $G \cp H$.  For a vertex $g$ of $G$, the subdigraph of $G\cp H$ induced by the set $\{(g,h) :\, h\in V(H)\}$ is an \emph{$H$-fiber} and is denoted by $^g\!H$.  Similarly, for $h \in V(H)$, the \emph{$G$-fiber}, $G^h$, is the subdigraph induced by $\{(g,h) :\, g\in V(G)\}$.  Note that every $G$-fiber is isomorphic to the digraph $G$, and every $H$-fiber is isomorphic to the digraph $H$.  We also use the fiber notation to refer to the set of vertices in these subgraphs; the meaning should be clear from the context.
The \emph{direct product} of digraphs $G$ and $H$, denoted $G \times H$, has $V(G) \times V(H)$ as its vertex set.  If $(g_1,h_1)$ and $(g_2,h_2)$ are vertices in $G \times H$, then $(g_1,h_1)$ is adjacent to $(g_2,h_2)$ (equivalently, $(g_2,h_2)$ is adjacent from $(g_1,h_1)$) if $g_1g_2 \in A(G)$ and $h_1h_2 \in A(H)$.  The fibers of the direct
product $G \times H$ are defined  as they are for the Cartesian product, but in the direct product they are independent sets. We remark in passing that a few recent papers considered domination in digraph products concentrating on efficient domination~\cite{pgy-2019}, rainbow domination~\cite{hmwx} and digraph kernels~\cite{lv-2016}.

The following facts follow immediately from the definitions.
\begin{obs} \label{obs:immediate}
Let $G$ be any digraph.
\begin{enumerate}
\item Since $N_{\ugr{G}}[v]=N^+_G[v] \cup N^-_G[v]$ for every $v \in V(G)$, we infer that $\rho_2(\ugr{G})\le \rho(G)$.
\item Since $N_{\ugr{G}}(v)=N^+_G(v) \cup N^-_G(v)$ for every $v \in V(G)$, we infer that $\opack(\ugr{G}) \leq \opack(G)$.
\item Since every dominating set of $G$ intersects every closed in-neighborhood of $G$, we get $\gamma(G) \ge \rho(G)$.
\item If $\delta^-(G) \ge 1$, then $\gamma_t(G) \ge \opack(G)$ since any total dominating set of $G$ intersects every open in-neighborhood of $G$.
\item Since $N^+_G[S]\subseteq N_{\ugr{G}}[S]$ for every $S \subseteq V(G)$, it follows that $\gamma(G) \ge \gamma(\ugr{G})$.
\end{enumerate}
\end{obs}

The following definitions refer only to undirected graphs $G$. A {\em clique} is a set of vertices in $G$ that induces a complete graph.  The maximum cardinality of an independent set of vertices in $G$ is the {\em independence number}, $\alpha(G)$, of $G$. The minimum cardinality of a partition of $V(G)$ into independent sets is the {\em chromatic number}, $\chi(G)$, of $G$. The maximum cardinality of a clique in $G$ is denoted by $\omega(G)$, and the minimum cardinality of a partition of $V(G)$ into cliques is denoted by $k(G)$. A graph $G$ is {\em perfect} if $\omega(H)=\chi(H)$ for every induced subgraph $H$ of $G$. By the {\bf Perfect Graph Theorem} of Lov\'{a}sz~\cite{lo-1972} this is equivalent to saying that $\alpha(H)=k(H)$ for every induced subgraph $H$ of $G$. In particular, the Perfect Graph Theorem implies that $G$ is perfect if and only if its complement $\overline{G}$ is perfect.

For a positive integer $n$ we denote the set of positive integers less than or equal to $n$ by $[n]$. If $G$ is a digraph or an undirected graph, then $n(G)$ will denote the order of $G$; that is, $n(G)=|V(G)|$.

%%%%%%%%%%%%%%%%%%%%%%%
\section{Packing and domination in graphs}
\label{sec:p-d-graphs}
%%%%%%%%%%%%%%%%%%%%%%%

It is clear that for an undirected graph $G$ it follows that $\gamma(G)$ is the minimum number of closed neighborhoods of vertices of $G$ whose union is $V(G)$.  For convenience we will say that this collection of closed neighborhoods \emph{covers} $G$.
For the undirected graph $G$ we define an undirected graph, the \emph{closed neighborhood graph} of $G$, denoted by $\cng(G)$ as follows.   The vertex set of
$\cng(G)$ is the set $V(G)$; two distinct vertices $u$ and $v$ in $V(G)$ are adjacent in $\cng(G)$ if and only if  $N_G[u] \cap N_G[v] \neq \emptyset$.  Note that $\cng(G)$
is isomorphic to $G^2$, the graph obtained from $G$ by adding an edge $uv$ for each pair of vertices with $d_G(u,v)=2$. It is known that the graph $G^2$, where $G$ is a tree, is a chordal graph, and thus a perfect graph~\cite{ha-mc-1994}.

An undirected graph $G$ is a {\em neighborhood Helly graph} if the closed neighborhoods $N_G[x]$ of vertices $x\in V(G)$ enjoy the {\em Helly property}. That is, any collection of closed neighborhoods in $G$, which are pairwise intersecting, have a common intersection; see~\cite{dps-2009,dragan}. It is well-known that trees are Helly graphs, which are in turn neighborhood Helly graphs.

The following statement follows directly from the definitions.
If $G$ is a graph and $S \subseteq V(G)$, then $S$ is a $2$-packing in $G$ if and only if $S$ is an independent set in $\cng(G)$.  Hence we have,

\begin{obs} \label{obs:graphsimmediate}
If $G$ is an undirected graph, then $\rho_2(G)=\alpha(\cng(G))$.
\end{obs}

We now give a proof of Theorem~\ref{thm:mm} that is different in approach from that of Meir and Moon in~\cite{mm-1975}.
\medskip

\noindent \textbf{Theorem~\ref{thm:mm}} \emph{
If $T$ is any tree, then $\rho_2(T)=\gamma(T)$.
}
\begin{proof}
Let $C=\{x_1, \ldots,x_m\}$ be a clique of $\cng(T)$. Therefore, the closed neighborhoods $N_T[x_1],\ldots,N_T[x_m]$ are pairwise intersecting, and since trees are Helly graphs, we infer that $N_T[x_1]\cap\cdots\cap N_T[x_m]\ne \emptyset$. Thus there exists $w\in N_T[x_1]\cap\cdots\cap N_T[x_m]$. We infer that $C\subseteq N_T[w]$.
Since $\cng(T)=T^2$ is a chordal graph, and hence a perfect graph, we conclude by the Perfect Graph Theorem that the complement $\overline{\cng(T)}$ is also a perfect graph.  In particular,  $\omega(\overline{\cng(T)})=\chi(\overline{\cng(T)})$.  Now, using Observation~\ref{obs:graphsimmediate}, we get
\[\rho_2(T)=\alpha(\cng(T))=\omega(\overline{\cng(T)})=\chi(\overline{\cng(T)})=\gamma(T)\,.\]
The last equality follows since $\chi(\overline{\cng(T)})$ is the minimum number of independent sets of $\overline{\cng(T)}$ that cover $\overline{\cng(T)}$.  This minimum number of  independent sets of $\overline{\cng(T)}$ that cover $\overline{\cng(T)}$ is the same as the minimum number of cliques that cover $\cng(T)$, which in turn is the minimum number of closed neighborhoods of $T$ that cover $T$.
\end{proof}

%%%%%%%%%%%%%%%%%%%%%%%%%%%%%%%%
\section{Packing and domination in digraphs} \label{sec:p-d-digraphs}
%%%%%%%%%%%%%%%%%%%%%%%%%%%%%%%%
Similarly as remarked for undirected graphs in Section~\ref{sec:p-d-graphs}, we note that the domination number of a digraph $D$ is the minimum number of closed out-neighborhoods of vertices of $D$ whose union is $V(D)$. As is the case for undirected graphs, we say this collection of closed out-neighborhoods \emph{covers} $D$.
For a digraph $D$ we define an undirected graph, the \emph{closed in-neighborhood graph} of $D$, denoted by $\cing(D)$.  The vertex set of $\cing(D)$ is $V(D)$ and two vertices of $\cing(D)$ are adjacent if and only if $N^-_D[u] \cap N^-_D[v] \neq \emptyset$.  We note that $S$ is a packing in $D$ if and only if $S$ is independent in $\cing(D)$. This yields the digraph analogue to Observation~\ref{obs:graphsimmediate}:

\begin{obs} \label{obs:immediatedigraph}
If $D$ is a digraph, then $\rho(D)=\alpha(\cing(D))$.
\end{obs}

We now prove a lemma in which  we restrict to digraphs whose underlying graph has no cycle of length less than $7$.  This lemma will then lead to a proof of
Theorem~\ref{thm:packing-dom-ditrees} that uses essentially the same approach as the proof of Theorem~\ref{thm:mm}.

\begin{lem} \label{lem:dicliques}
Let $D$ be a digraph such that $\ugr{D}$ has girth at least $7$.  If $K$ is a maximal clique of $\cing(D)$, then there is a vertex $w$ in $D$ such that
$K \subseteq N_D^+[w]$.
\end{lem}
\begin{proof}

First we claim that $|N_D^-[u] \cap N_D^-[v]| \leq 2$, for any two distinct vertices $u$ and $v$ of $G$.  If $uv \in A(D)$ or $vu \in A(D)$  and there exists $x \in (N_D^-[u] \cap N_D^-[v])-\{u,v\}$, then $\{x,u,v\}$ induces a triangle in $\ugr{D}$, which is a contradiction.  On the other hand, if neither $uv$ nor $vu$ is an arc in $D$, then $|N_D^-[u] \cap N_D^-[v]| \leq 1$.  For if $y$ and $z$ are distinct and $\{y,z\} \subseteq N_D^-[u] \cap N_D^-[v]$, then the subgraph of $\ugr{D}$ induced by $\{u,v,y,z\}$ contains a $4$-cycle, which is again a contradiction.

Let $K=\{x_1, \ldots,x_m\}$ be a maximal complete subgraph of $\cing(D)$.  Suppose first that $m=2$.  Since $x_1x_2$ is an edge in $\cing(D)$, it follows that there exists a vertex $w \in N_D^-[x_1] \cap N_D^-[x_2]$.  This implies that $K=\{x_1,x_2\}\subseteq N_D^+[w]$.  Thus, we now assume that $m \ge 3$.  Let $i,j$ and $k$ be any three distinct indices in $[m]$.
The vertices $x_i,x_j$ and $x_k$ are pairwise adjacent in $\cing(D)$, which implies that there exist (not necessarily distinct) vertices $a$, $b$ and $c$ in $G$  such that
\[ a\in N_D^-[x_i] \cap N_D^-[x_j], \hskip .5cm b\in N_D^-[x_j] \cap N_D^-[x_k], \text{    and    } c\in N_D^-[x_i] \cap N_D^-[x_k]\,.\]
 Let $e$ be the number of edges in the subgraph of $\ugr{D}$ induced by
$\{x_i,x_j,x_k\}$.  We claim that $e \in \{0,2\}$.  Since the girth of $\ugr{D}$ is at least $7$, it is clear that $e\neq 3$.  Suppose $e=1$.  Without loss of generality we assume that $x_ix_j \in A(D)$.  In this case it follows that $b \neq c$ because $\ugr{D}$ has no triangles.  Since $|N_D^-[x_i] \cap N_D^-[x_j]| \leq 2$, it follows that $a\in \{x_i,x_j\}$. However, this now implies that $\{x_i,x_j,x_k,b,c\}$ induces a subgraph in $\ugr{D}$ that contains a $5$-cycle, which is a contradiction.  Therefore, $e=0$ or $e=2$.  That is, for any three vertices from $K$, the subgraph of $\ugr{D}$ they induce has either zero or two edges.

Suppose first that there exists  three distinct vertices in $K$ that induce a $P_3$ in $\ugr{D}$.  Without loss of generality suppose $x_1x_2$ and $x_1x_3$ are edges in $\ugr{D}$.  Since $x_2x_3 \in E(\cing(D))$, there exists $b \in N_D^-[x_2]\cap N_D^-[x_3]$. If $b \not\in \{x_1, x_2, x_3\}$, then $\ugr{D}$ contains a $4$-cycle. Thus, $b \in \{x_1, x_2, x_3\}$. However, $b \not\in \{x_2, x_3\}$ for otherwise $\ugr{D}$ contains a triangle. Thus, $x_1x_2 \in A(D)$ and $x_1x_3\in A(D)$.  If $m=3$, then we have established the lemma by setting $w=x_1$  since $K=\{x_1,x_2,x_3\}\subseteq N_D^+[x_1]$.  Suppose then that $m>3$.  Let $t \in [m]-\{1,2,3\}$ and consider the set $\{x_1,x_2,x_t\}$.  By the above we infer that this set induces a subgraph of $\ugr{D}$ that contains exactly two edges.  Suppose $x_tx_2 \in A(D)$ or $x_2x_t \in A(D)$.  The girth of $\ugr{D}$ is at least $7$, and this implies $x_tx_3 \notin E(\ugr{D})$.
However, $x_tx_3 \in E(\cing(D))$, and hence there exists a vertex $d\in (N_D^-[x_t] \cap N_D^-[x_3])-\{x_t,x_3\}$.  This implies that the subgraph of $\ugr{D}$ induced by $\{x_1,x_2,x_3,x_t,d\}$ contains a $5$-cycle, which is a contradiction.  But this means that $x_tx_1 \in A(D)$ or $x_1x_t \in A(D)$.  Therefore, if $K$ contains a triple that induces a $P_3$ in $\ugr{D}$, then $K$ induces a star in  $\ugr{D}$.  Let $2 \le i<j\le m$.  Since $x_ix_j \in E(\cing(D))$, there exists $b \in N_D^-[x_i]\cap N_D^-[x_j]$.
We know that $b \not\in \{x_2, \dots, x_m\}$, for otherwise $\ugr{D}$ contains a $3$- or $4$-cycle. On the other hand, if $b \not\in \{x_1, \dots, x_m\}$, then $\ugr{D}$ contains a $4$-cycle. Thus, $b = x_1$ and $x_1x_i \in A(D)$ for $2 \le i \le m$.  That is, $x_i \in N_D^+[x_1]$, for every $i \in [n]$.  Setting $w = x_1$ establishes the lemma.

Finally, suppose every subset of $K$ of cardinality $3$ is an independent set in $\ugr{D}$.  We treat $\{x_i,x_j,x_k\}$ using the notation set out in the second paragraph of this proof. If $|\{a,b,c\}|=3$, then $\ugr{D}$ contains the $6$-cycle $x_i,a,x_j,b,x_k,c,x_i$, which is a contradiction.  If $|\{a,b,c\}|=2$, say $a=b$, then $a,x_i,c,x_k,a$ is a cycle in $\ugr{D}$, which is again a contradiction.  Therefore, $a=b=c$.  The set $S=\{a,x_i,x_j,x_k\}$ induces a star in $\ugr{D}$  with center $a$, and $S$ induces a complete subgraph of $\cing(D)$.  Since $K$ is a maximal complete subgraph of $\cing(D)$, we conclude that if $m=3$, then $\{x_1,x_2,x_3\}$ is not independent in $\ugr{D}$.
By the previous case we conclude that $e=2$ and the conclusion of the lemma holds.  Thus we assume $m \ge 4$. If we can show that $a \in \cap_{s=1}^m N_{\ugr{D}}[x_s]$, then by
the maximality of $K$ it will follow that $a\in K$.  Let $t \in [m]-\{x_i,x_j,x_k\}$ such that $x_t \neq a$ and again consider the set $\{x_i,x_j,x_t\}$, which by assumption is independent in $\ugr{D}$.  Suppose $ax_t\notin E(\ugr{D})$.  Since $x_ix_t$ and $x_jx_t$ are edges in $\cing(D)$ it follows that there exist $u \in (N_D^-[x_t] \cap N_D^-[x_i])-\{x_t,x_i\}$ and $v \in (N_D^-[x_t] \cap N_D^-[x_j])-\{x_t,x_j\}$.  This implies that $a,x_i,u,x_t,v,x_j,a$ is a $6$-cycle in $\ugr{D}$.  This is a contradiction, and hence $ax_t \in E(\ugr{D})$.  Since $t$ was arbitrary, we conclude that $a \in K$.  Consequently, it is not possible that every subset of $K$ having cardinality $3$ is independent in $\ugr{D}$.  This contradiction finishes the proof.
\end{proof}

We are now able to prove that in any ditree the domination number and the packing number are equal.  Suppose $T$ is any ditree.  Recall that the domination number is the cardinality of a smallest subset $S$ of $V(T)$ such that the  closed out-neighborhoods of the vertices in $S$ cover $T$.  The packing number of $T$ is the cardinality of a largest subset $P$ of $V(T)$ such that the closed in-neighborhoods of vertices in $P$ are pairwise disjoint.

\bigskip

\noindent \textbf{Theorem~\ref{thm:packing-dom-ditrees}} \emph{
For any ditree $T$, $\rho(T) = \gamma(T)$.
}

\begin{proof}
Let $T$ be a ditree.
We claim that $\cing(T)$ is a chordal graph.  Suppose $C: v_1,v_2,\ldots,v_n,v_1$ is a chordless cycle in $\cing(T)$ for some $n \ge 4$.  We compute  subscripts modulo $n$.
For each $i \in [n]$ the edge $v_iv_{i+1}$ is in  $\cing(T)$.  This implies that $N_T^-[v_i] \cap N_T^-[v_{i+1}] \neq \emptyset$.  Let $a_i \in N_T^-[v_i] \cap N_T^-[v_{i+1}]$.  That is, either $\{v_iv_{i+1}, v_{i+1}v_i\} \cap A(T) \neq \emptyset$ (in which case $a_i=v_i$ or $a_i=v_{i+1}$) or there exists $a_i \in V(T)-\{v_i,v_{i+1}\}$ such that $a_iv_i$ and $a_iv_{i+1}$ are both arcs of $T$.  Since $C$ is chordless, any two $a_r, a_s$ that do not belong to $C$ are distinct whenever $r\ne s$.  However, this means that $\ugr{T}$ contains a cycle, which is a contradiction.  Therefore, $\cing(T)$ is a chordal graph and hence is a perfect graph.  By the Perfect Graph Theorem of Lov\'{a}sz, we conclude that the complement $\overline{\cing(T)}$ is also a perfect graph.  In particular,
$\omega(\overline{\cing(T)})=\chi(\overline{\cing(T)})$.  Now we get
\[\rho(T)=\alpha(\cing(T))=\omega(\overline{\cing(T)})=\chi(\overline{\cing(T)})=\gamma(T)\,.\]
The last equality follows since $\chi(\overline{\cing(T)})$ is the minimum number of independent sets of $\overline{\cing(T)}$ that cover $\overline{\cing(T)}$.  This minimum number of  independent sets of $\overline{\cing(T)}$ that cover $\overline{\cing(T)}$ is the same as the minimum number of maximal cliques that cover $\cing(T)$, which in turn, by Lemma~\ref{lem:dicliques} is the minimum number of closed out-neighborhoods of $T$ that cover $T$.
\end{proof}

We conclude this section with the following question.

\begin{prob} \label{prob:acyclic}
Let $G$ be an acyclic digraph (that is, there are no directed cycles in $G$). Is it true that $\rho(G)=\gamma(G)$?
\end{prob}

%%%%%%%%%%%%%%%%%%%%%%%%%%%%%%%%
\section{Open packing and total domination in digraphs} \label{sec:op-td-digraphs}
%%%%%%%%%%%%%%%%%%%%%%%%%%%%%%%%

In this section we explore the relationship between open packings and total domination in digraphs.  We will thus assume that all digraphs under discussion have minimum
in-degree at least $1$.  First we note that the total domination number of a digraph $D$ is the minimum number of open out-neighborhoods of vertices of $D$ whose union is $V(D)$. As in Section~\ref{sec:p-d-digraphs}, we say this collection of open out-neighborhoods \emph{covers} $D$.  For a digraph $D$ we define a second undirected graph, the \emph{open in-neighborhood graph} of $D$, denoted by $\oing(D)$.  The vertex set of $\oing(D)$ is $V(D)$ and $E(\oing(D))=\{uv :\, N_D^-(u) \cap N_D^-(v) \neq \emptyset \}$. We make the following observation that follows directly from the definitions.

\begin{obs} \label{obs:immediatedigraphopen}
If $D$ is a digraph, then $\opack(D)=\alpha(\oing(D))$.
\end{obs}

\begin{lem} \label{lem:dicliquesopen}
Let $D$ be a digraph such that $\ugr{D}$ has girth at least $7$.  If $K$ is a clique in $\oing(D)$, then there is a vertex $w$ in $D$ such that
$K \subseteq N_D^+(w)$.
\end{lem}

\begin{proof}
Let $K=\{x_1,\ldots,x_m\}$ be a complete subgraph of $\oing(D)$.  By definition, $N_D^-(x_i) \cap N_D^-(x_j) \neq \emptyset$ for every pair of indices $i$ and $j$ in $[m]$.
Suppose first that $m=2$.  In this case for any $w \in N_D^-(x_1) \cap N_D^-(x_2)$, we have $wx_1$ and $wx_2$ are arcs in $D$, and thus $K \subseteq N_D^+(w)$.  Now, assume
that $m \ge 3$.  There exist (not necessarily distinct) vertices $a,b,c$ such that $a \in N_D^-(x_1) \cap N_D^-(x_2)$, $b \in N_D^-(x_2) \cap N_D^-(x_3)$, and
$c \in N_D^-(x_3) \cap N_D^-(x_1)$.  That is, $\{ax_1,ax_2,bx_2,bx_3,cx_3,cx_1\} \subseteq A(D)$. Suppose that $\{a,b,c\} \cap \{x_1,x_2,x_3\} \neq \emptyset$.  Without loss of
generality assume that $a \in \{x_1,x_2,x_3\}$.  Since $D$ has no loops, it follows that $a=x_3$. It now follows that $\{b,x_2,x_3\}$ induces a triangle in $\ugr{D}$, which is a
contradiction.  Therefore, $\{a,b,c\} \cap \{x_1,x_2,x_3\} = \emptyset$.  If $a,b$ and $c$ are three distinct vertices, then the subgraph of $\ugr{D}$ induced by
$\{a,b,c,x_1,x_2,x_3\}$ contains a cycle of length at most $6$.  With no loss of generality we assume then that $a=b$.  This implies that the digraph $D$ contains all of the arcs $ax_1, ax_2$ and $ax_3$.  If, in addition, $b \neq c$, then $\ugr{D}$ contains the $4$-cycle $a,x_1,c,x_3,a$, which is a contradiction.  We conclude that $a=b=c$, and
hence $\{x_1,x_2,x_3\} \subseteq N_D^+(a)$.  If $m\ge 4$, then for any $i \in [m]-[3]$ consider the subset $\{x_1,x_2,x_i\} \subseteq K$.  Using
the same argument given above for $\{x_1,x_2,x_3\}$, we can conclude that $\{x_1,x_2,x_i\} \subseteq N_D^+(a')$ for some $a'\in V(D)$.  This implies that $a=a'$, for otherwise
$\ugr{D}$ contains the $4$-cycle $a,x_1,a',x_2,a$.  Therefore, $K \subseteq N_D^+(a)$.
\end{proof}

With Lemma~\ref{lem:dicliquesopen} in hand we are now able to prove a result for ditrees that is analogous to the result of Rall~\cite[Lemma 10]{ra-2005} that showed every nontrivial (undirected) tree has open packing number equal to its total domination number.  We need to assume the ditrees in our result have minimum in-degree at least $1$, which is required in order to have a total dominating set.  We use the fact that the total domination number of such a ditree $T$ is the minimum number of open out-neighborhoods in $T$  that cover $V(T)$.  By Lemma~\ref{lem:dicliquesopen} this is the same as the minimum number of complete subgraphs of $\oing(T)$ that cover $V(T)$.

\bigskip

\noindent \textbf{Theorem~\ref{thm:openpacking-tdom-ditrees}} \emph{
If $T$ is a ditree such that $\delta^-(T) \ge 1$, then $\opack(T) = \gamma_t(T)$.
}
\begin{proof}
Let $T$ be a ditree such that $\delta^-(T) \ge 1$.  By Observation~\ref{obs:immediate}, we have $\gamma_t(T) \ge \opack(T)$.  We claim that the chromatic number and the clique number of the complement of the open in-neighborhood graph $\oing(T)$ are equal.  To verify this we will first show that $\oing(T)$ is a chordal graph.  Suppose
$C: v_1,v_2,\ldots,v_n,v_1$ is a chordless cycle in $\oing(T)$ for some $n\ge 4$.  We compute subscripts of the vertices in $C$ modulo $n$.  For each $i \in [n]$, there exists
a vertex $a_i\in N_T^-(v_i) \cap N_T^-(v_{i+1})$ since $v_iv_{i+1}$ is an edge in $\oing(T)$.  This yields a closed walk $W:v_1,a_1,v_2,\ldots,v_i,a_i,v_{i+1},\ldots,a_n,v_1$
in $\ugr{T}$.  Suppose that $a_s=a_t$ for some $1 \le s <t \le n$.  If $t-s=1$, then $\{a_sv_s,a_sv_{s+1},a_sv_{s+2}\} \subseteq A(T)$.  This implies that $v_sv_{s+2}$ is a
chord of $C$, which is a contradiction.  On the other hand, suppose that $2 \le t-s < n-1$.  In this case $\{a_sv_s,a_sv_{s+1},a_sv_t,a_sv_{t+1}\} \subseteq A(T)$, which implies
that $v_sv_t$ is a chord of $C$.  We conclude that $W$ is a cycle in $\ugr{T}$, which is a contradiction.  Therefore, $\oing(T)$ is a chordal graph and thus is a perfect graph.
The Perfect Graph Theorem implies that $\overline{\oing(T)}$ is also a perfect graph.

Using Lemma~\ref{lem:dicliquesopen} and in a manner similar to the proof of Theorem~\ref{thm:packing-dom-ditrees}, we get
\[\opack(T)=\alpha(\oing(T))=\omega(\overline{\oing(T)})=\chi(\overline{\oing(T)})=\gamma_t(T)\,.\]
\end{proof}

It is easy to see that $\gamma_t(G \times H) \le \gamma_t(G)\gamma_t(H)$, for any pair of digraphs $G$ and $H$ that both have minimum in-degree at least $1$.  This follows
from the fact that if $A$ and $B$ are minimum total dominating sets of $G$ and $H$, respectively, then $A \times B$ totally dominates $G \times H$.  In addition, for such
a pair of digraphs, it follows from the definitions that
$\gamma_t(G \times H) \ge \max\{\opack(G)\gamma_t(H),\opack(H)\gamma_t(G)\}$.
We can now use Theorem~\ref{thm:openpacking-tdom-ditrees} to specify many instances of pairs of digraphs that achieve equality in the above.  This is the digraph version of
the result of Rall~\cite[Theorem 7]{ra-2005} for undirected graphs.

\bigskip

\noindent \textbf{Theorem~\ref{thm:equaltotalfordirect}} \emph{
If $G$ is a digraph with $\delta^-(G) \ge 1$ such that $\opack(G)=\gamma_t(G)$, then $\gamma_t(G \times H)=\gamma_t(G)\gamma_t(H)$, for every digraph $H$ with minimum in-degree at least $1$.
}

\begin{proof}
Suppose $G$ satisfies the hypothesis of the theorem and let $H$ be a digraph with $\delta^-(H) \ge 1$.  Now,
\[\gamma_t(G)\gamma_t(H) \ge \gamma_t(G \times H) \ge \max\{\opack(G)\gamma_t(H),\opack(H)\gamma_t(G)\} \ge \opack(G)\gamma_t(H)=\gamma_t(G)\gamma_t(H)\,.\]
\end{proof}

The next result follows immediately by applying Theorem~\ref{thm:openpacking-tdom-ditrees} and Theorem~\ref{thm:equaltotalfordirect}.

\begin{cor} \label{cor:equaltyditreetdomdirect}
If $T$ is any ditree with minimum in-degree at least $1$ and $H$ is any digraph with $\delta^-(H) \ge 1$, then
\[\gamma_t(T \times H)=\gamma_t(T)\gamma_t(H)\,.\]
\end{cor}

Somewhat surprisingly there are pairs of digraphs whose direct product requires all of the vertex set to totally dominate.  For example, for a positive integer $n$
let $G_n$ be the digraph formed from the cycle $C_n$ by orienting its edges so that every vertex has in-degree $1$ and out-degree $1$.  It follows that
$\opack(G_n)=\gamma_t(G_n)=n$.  Thus for positive integers $n$ and $m$ we get by Theorem~\ref{thm:equaltotalfordirect} that $\gamma_t(G_m \times G_n)=mn$.

%%%%%%%%%%%%%%%%%%%%%%%%%
\section{Lower bounds on $\gamma(G\cp H)$} \label{sec:results}
%%%%%%%%%%%%%%%%%%%%%%%%%

Recall the proposition announced in the beginning of the paper.

\medskip

\noindent \textbf{Proposition~\ref{prp:packing}} \emph{
If $G$ and $H$ are digraphs, then
\[
\gamma(G \cp H) \ge \max\{\gamma(G)\rho(H),\gamma(H)\rho(G)\}\,.
\]
}
The proposition can be proved as follows. Given a packing $P=\{p_1,\ldots,p_{\rho(G)}\}$ of $G$, partition a subset of $V(G)\times V(H)$ into the blocks $N_G^-[p_i]\times V(H)$ where $i\in  [\rho(G)]$. These blocks are indeed pairwise disjoint, since $P$ is a packing.  Note that only vertices of $N_G^-[p_i]\times V(H)$ can be used to dominate vertices in the $H$-fiber $^{p_i}\!H$, and at least $\gamma(H)$ of such vertices are needed to dominate $^{p_i}\!H$. This gives $\gamma(G \cp H) \ge \gamma(H)\rho(G)$, and for the other inequality reverse the roles of $G$ and $H$.

Combining Proposition~\ref{prp:packing} and Theorem~\ref{thm:packing-dom-ditrees}, we infer that all ditrees satisfy Vizing's inequality:

\begin{cor}
\label{cor:allditrees}
If $T$ is a ditree and $H$ is a digraph, then
\[
\gamma(T \cp H) \ge \gamma(T)\gamma(H)\,.
\]
\end{cor}

Proposition~\ref{prp:packing} and Corollary~\ref{cor:allditrees} are extensions of results of Jacobson and Kinch~\cite[Theorem 2 and Corollary 3]{JK} from undirected graphs to digraphs.

It is known that the inequality analogous to~\eqref{VCeqn} does not hold in the context of digraphs; see e.g.~\cite{nca-2009}. For a small sporadic example of digraphs that do not satisfy Vizing's inequality, consider the digraphs $G_1$ and $H$ in
Figure~\ref{fig:two-digraphs}. Note that $\gamma(G_1)=2,\gamma(H)=3$, while $\gamma(G_1\Box H)=5$. To see $\gamma(G_1\Box H)\le 5$ note that the set $\{(c,u),(a,v),(c,x),(a,y),(b,z)\}$ is a dominating set of $G_1\Box H$.

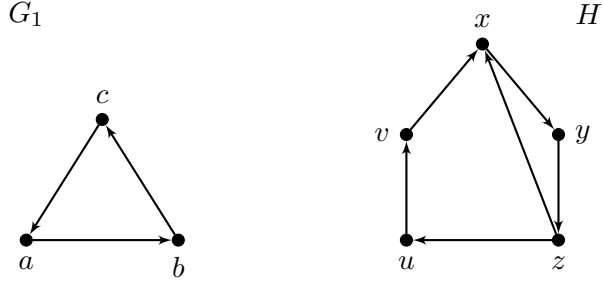
\begin{figure}[h!]
\begin{center}
\begin{tikzpicture}[auto]
\tikzset{edge/.style = {->,> = latex'}}
    % Place nodes
    %triangle
	\vertex (a) at (-4,0) [scale=.75pt,fill=black,label=below:$a$]{};
	\vertex (b) at (-2,0) [scale=.75pt,fill=black,label=below:$b$]{};
	\vertex (c) at (-3,1.6) [scale=.75pt,fill=black,label=above:$c$]{};
	% digraph H
	\vertex (u) at (1,0) [scale=.75pt,fill=black,label=below:$u$]{};
	\vertex (v) at (1,1.4) [scale=.75pt,fill=black,label=left:$v$]{};
	\vertex (x) at (2,2.6) [scale=.75pt,fill=black,label=above:$x$]{};
	\vertex (y) at (3,1.4) [scale=.75pt,fill=black,label=right:$y$]{};
	\vertex (z) at (3,0) [scale=.75pt,fill=black,label=below:$z$]{};
	
	\draw(-4,3)node{$G_1$};
    \draw(3.4,3)node{$H$};
	\draw[edge, thick] (a) to (b);
	\draw[edge, thick] (b) to (c);
	\draw[edge, thick] (c) to (a);
	\draw[edge, thick] (u) to (v);
	\draw[edge, thick] (v) to (x);
	\draw[edge, thick] (x) to (y);
	\draw[edge, thick] (y) to (z);
	\draw[edge, thick] (z) to (u);
	\draw[edge, thick] (z) to (x);

\end{tikzpicture}
\end{center}
\caption{The digraphs $G_1$ and $H$.}
\label{fig:two-digraphs}
\end{figure}

More generally, consider the following infinite family of examples that do not satisfy Vizing's inequality.  For a positive integer $m$ the digraph $G_m$ has vertex set $V(G_m)=\{v_1,v_2,v_3,\ldots,v_{2m},v_{2m+1}\}$ and arc set $A(G_m)=\{v_1v_{2i} :\, i \in [m]\} \cup \{v_{2i}v_{2i+1}:\, i \in [m]\} \cup \{v_{2i+1}v_1:\, i \in [m]\}$. It is easy to see that $\rho(G_m)=m$ and $\gamma(G_m)=m+1$.  Figure~\ref{fig:noVC} illustrates $G_3$, and $G_1$ is the digraph on the left in Figure~\ref{fig:two-digraphs}.

\begin{figure}[h!]
\begin{center}
\begin{tikzpicture}[auto]
\tikzset{edge/.style = {->,> = latex'}}
    % Place nodes
    \vertex (0) at (0,2) [scale=1pt,fill=black,label=above:$v_1$]{};
	\vertex (1) at (-5,0) [scale=1pt,fill=black,label=below:$v_2$]{};
	\vertex (2) at (-3,0) [scale=1pt,fill=black, label=below:$v_3$]{};
	\vertex (3) at (-1,0) [scale=1pt,fill=black, label=below:$v_4$]{};
    \vertex (4) at (1,0) [scale=1pt,fill=black, label=below:$v_5$]{};
    \vertex (5) at (3,0) [scale=1pt,fill=black, label=below:$v_6$]{};
    \vertex (6) at (5,0) [scale=1pt,fill=black, label=below:$v_7$]{};
    \draw[edge, thick] (0) to (1);
    \draw[edge, thick] (0) to (3);
    \draw[edge, thick] (0) to (5);
	\draw[edge, thick] (1) to (2);
	\draw[edge, thick] (2) to (0);
    \draw[edge, thick] (3) to (4);
    \draw[edge, thick] (4) to (0);
    \draw[edge, thick] (5) to (6);
    \draw[edge, thick] (6) to (0);
\end{tikzpicture}
\end{center}
\caption{The digraph $G_3$.}
\label{fig:noVC}
\end{figure}
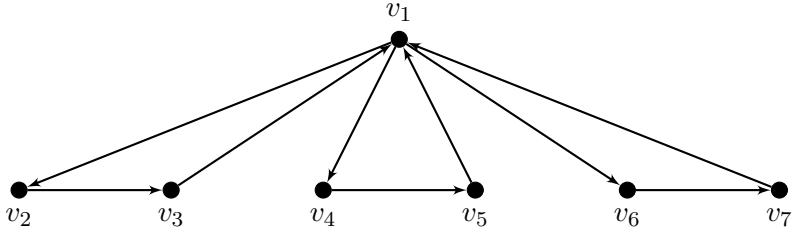

It can be verified that the set $S$, defined by
\[S=\{(v_1,v_{2i}):\, i\in [m]\} \cup \{(v_{2i},v_{2j+1}):\, i \in [m] \text{ and } j\in [m]\} \cup \{(v_{2j+1},v_1):\, j\in [m]\}\,,\]
dominates $G_m \cp G_m$.  Since $|S|=m+m(m+1)$, it follows that
\[\gamma(G_m \cp G_m)\le |S|=m^2+2m < \gamma(G_m)\gamma(G_m)\,.\]

Next, we present a different lower bound on $\gamma(G\cp H)$ expressed in terms of the domination numbers of digraphs $G$ and $H$, and prove that the inequality is sharp. The lower bound is a digraph version of the theorem proved by Zerbib~\cite{Zerbib}, which improves earlier results of Clark and Suen~\cite{ClSu20} and Suen and Tarr~\cite{ST2012}. The proof follows the idea of a framework presented in~\cite{BHHKR2020+}, which works also in the general (directed) case.

\begin{thm} \label{thm:Zerbib}
If $G$ and $H$ are digraphs, then
\[
\gamma(G \cp H) \ge \frac{1}{2}\gamma(G)\gamma(H) + \frac{1}{2}\max\{\gamma(G),\gamma(H)\}\,,
\]
and the bound is sharp.
\end{thm}
\begin{proof}
We may assume with no loss of generality that $\gamma(G) \ge \gamma(H)$. Let
$S$ be a minimum dominating set of $G \cp H$, and let
 $S_G$ be the projection $p_G(S)$ of the set $S$ to $G$. Since $S$ is a (minimum) dominating set of $G \cp H$, the set $S_G$ is a dominating set of $G$. Among all subsets of vertices of $S_G$ that form a dominating set of $G$, let $U = \{u_1, \ldots, u_k\}$ be one of minimum size. Since $U$ is a dominating set of $G$, we have $|U|=k\ge \gamma(G)$.

Consider a partition $\pi =\{\pi_1,\ldots,\pi_k\}$ of $V(G)$ chosen so that $u_i \in \pi_i$ and $\pi_i \subseteq N^+[u_i]$ for each $i$.  Let $H_i= \pi_i \times V(H)$.
For a vertex $h$ of $H$ the set of vertices $\pi_i\times\{h\}$ is a \emph{cell}, and is denoted shortly by $\pi_i^h$.

The cell $\pi_i^h$ is \emph{blue} if $\pi_i^h\cap S\neq\emptyset$ and $\pi_i^h$ is dominated by $S^h=S\cap G^h$.  The cell $\pi_i^h$ is \emph{green} if $\pi_i^h\cap S\neq\emptyset$ and $\pi_i^h$ is not dominated by $S^h$.
Finally, the cell is \emph{red} if it is dominated by $S^h$, and no vertex of $\pi_i^h$ is dominated by $S_i=S\cap V(H_i)$ (in particular, a red cell contains no vertex of $S$). All the remaining cells in $G^h$ are colored \emph{white}.  Note that exactly the cells with color blue and green contain vertices of $S$. We call vertices in $S$ that belong to the cell $\pi_i^h$ \emph{blue}, respectively \emph{green}, if the cell $\pi_i^h$ is blue, respectively green.
In the projection $p_H(H_i)$ of the cells of $H_i$ to $H$, the vertex $h = p_H(\pi_i^h)$ receives the color of the cell $\pi_i^h$ for each $i \in [k]$.

We also introduce the following notation concerning the size of some of the sets having certain colors.  Let $b_h'$ be the number of blue cells in the
fiber $G^h$, let $b_i'$ the number of blue cells in $H_i$, and let $b'$ the total number of all blue cells in $G\cp H$. We define analogously $g_h'$, $g_i'$, $g'$,
associated with the green cells  and $r_h'$, $r_i'$, and $r'$, associated with red cells.
Next, let $b_h$ and $b_i$ denote the number of blue vertices
in $G^h$ and $H_i$, respectively, and let $b$ be the total number of blue vertices in $G \cp  H$. In an analogous way we define $g_h$, $g_i$ and $g$, associated with the green vertices. Clearly, $|S|=b+g$. Since each blue cell contains at least one blue vertex, we get $b_h\ge b_h'$, $b_i\ge b_i'$ and $b \ge b'$, and analogously, $g_h \ge g_h'$, $g_i\ge g_i'$ and $g \ge g'$.

We need two auxiliary claims.

\begin{claim}
\label{c:claim1}
$b'+g'+r' \ge k \gamma(H)$.
\end{claim}
\begin{proof}
For any $i\in [k]$, we consider the cells of $H_i$. Note that the vertex $h=p_H(\pi_i^h)$ receives the color of the cell $\pi_i^h$. We claim that the set $P$ of vertices in $H$ that received (by the projection) one of the colors blue, green or red, forms a dominating set of $H$. Indeed, consider a white vertex $h\in V(H)$ that is projected from a white cell $\pi_i^h$.  By definition the white cell $\pi_i^h$ contains at least one vertex that is dominated by $S_i$. More precisely, $\pi_i^h$ contains a vertex $(u,h)$ that is dominated by a vertex $(u,h')\in S$, which is clearly blue or green. Hence $h'=p_H(\pi_i^{h'})$ received color blue or green and it dominates $h$ in $H$, as claimed.  Therefore, $P$ is a dominating set of $H$ and it is clear that $|P|=b_i'+g_i'+r_i'$, which implies $b_i'+g_i'+r_i'\ge \gamma(H)$. Summing up over all $i$ between $1$ and $k$, we get
$b'+g'+r' = \sum_{i=1}^{k}{(b_i'+g_i'+r_i')} \ge k\gamma(H)$.
\end{proof}

\begin{claim}
\label{c:claim2}
$r'\le b-b'+g$.
\end{claim}
\begin{proof}
For $h \in V(H)$, consider the blue and green vertices in the fiber $G^h$. These $b_h + g_h$ vertices dominate all vertices in the red and blue cells in
$G^h$. In the projection $p_G(G^h)$ of $G^h$ onto $G$ we note that the vertices projected from the blue and green vertices, together with the vertices $u_i$
from every projected white and green cell $\pi^h_i$, form a dominating set $U_h$ of $G$. Formally,
\[
U_h = p_G(S \cap G^h) \cup \{u_i \in U \mid \pi^h_i \mbox{ is a white or green cell}\}.
\]
Since the set $U_h$ is a dominating set of $G$ and since $U_h \subseteq S_G$, by the minimality of $U$ we have $|U| \le |U_y|$. Thus, $k = |U| \le |U_h| = b_h + g_h + (k - b_h' - r_h')$, or, equivalently, $r_h' \le b_h - b_h' + g_h$. Summing over all vertices $h \in V(H)$, we infer that $r' \le b-b'+g$.
\end{proof}

We now return to the proof of Theorem~\ref{thm:Zerbib}. For each vertex $u_i \in U$, let $(u_i,h)$ be a vertex of $S$ that belongs to the $H$-fiber, $^{u_i}\!H$. Since the cell $\pi_i^h$ is a blue cell, there are at least $k$ blue cells, implying that $b \ge k$. By Claim~\ref{c:claim1} and~\ref{c:claim2}, the following holds.
\[
\begin{array}{lcl}
\gamma(G)+\gamma(G)\gamma(H) & \le &
k + k \gamma(H) \\
& \le & b + k \gamma(H) \\
& \stackrel{(Claim~\ref{c:claim1})}{\le} & b + (b' + g' + r') \\
& \stackrel{(Claim~\ref{c:claim2})}{\le} & b + (b' + g' + (b - b' + g)) \\
& = & 2b + g + g' \\
& \le & 2(b + g) \\
& = & 2|S|,
\end{array}
\]
implying that $\gamma(G \cp H) = |S| \ge \frac{1}{2} \gamma(G)\gamma(H) + \frac{1}{2}\max\{\gamma(G),\gamma(H)\}$.

To see the sharpness of the bound, consider the Cartesian product of the strongly connected oriented triangle $T$ with itself. Note that $\gamma(T)=2$, while $\gamma(T\cp T)=3=\frac{1}{2}\gamma(T)^2+ \frac{1}{2}\max\{\gamma(T),\gamma(T)\}$.
\end{proof}

The following example gives an infinite family of strongly connected digraphs that attain the bound in Theorem~\ref{thm:Zerbib}. They also show that $\gamma(G)\gamma(H)-\gamma(G\cp H)$ can be arbitrarily large. Let $m=3k$, where $k\ge 3$, and for every $i\in [m]$ let $T_i$ be a copy of the strongly connected oriented triangle $G_1$ from Figure~\ref{fig:two-digraphs}. Let $V(T_i)=\{a_i,b_i,c_i\}$, and let $H_m$ be the digraph with $$V(H_m)=\bigcup_{i\in m}{V(T_i)}\cup \{d_1,\ldots,d_m\},$$
and $A(H_m)$ defined as follows. Let $$A'=\bigcup_{i=1}^k{\{a_{3i-2}d_{3i-2},a_{3i-2}d_{3i-1},a_{3i-2}d_{3i},b_{3i-1}d_{3i-2},b_{3i-1}d_{3i-1},b_{3i-1}d_{3i},c_{3i}d_{3i-2},c_{3i}d_{3i-1},c_{3i}d_{3i}\}},$$
$$A''=\{d_ia_{i+3}:\, i\in [m]\textrm{ and }i\textrm{ is taken modulo }m\},$$
and, $$A(H_m)=A'\cup A''\cup \bigcup_{i\in m}{A(T_i)}.$$
Note that $\gamma(H_m)=2m$. Now, taking the strongly connected oriented triangle $G_1$ with $V(G_1)=\{a,b,c\}$, the set $$D= \{(a_i,a):\,i\in [m]\}\cup\{(b_i,b):\,i\in [m]\}\cup \{(c_i,c),i\in [m]\}$$
is a dominating set of $H_m\cp G_1$ with size $|D|=3m$. This implies $$3m\ge \gamma(H_m\cp G_1)\ge \frac{1}{2}\gamma(H_m)\gamma(G_1)+ \frac{1}{2}\max\{\gamma(H_m),\gamma(G_1)\}=\frac{1}{2}(4m)+\frac{1}{2}(2m)=3m.$$

The following problem is natural, but may be difficult.

\begin{prob}
Characterize the (pairs of) digraphs $G$ and $H$ that attain the equality in the inequality of Theorem~\ref{thm:Zerbib}. That is, $\gamma(G \cp H)=\frac{1}{2}\gamma(G)\gamma(H) + \frac{1}{2}\max\{\gamma(G),\gamma(H)\}$.
\end{prob}

We conclude the section with the following easy observation, which extends from an analogous result in undirected graphs.

\begin{obs} \label{obs:maxbound}
If $G$ is a digraph with maximum out-degree $\Delta$, then $\gamma(G) \ge \frac{n(G)}{\Delta+1}$.
\end{obs}

As an easy application of Observation~\ref{obs:maxbound} let $\overrightarrow{C_m}$ be the strongly connected oriented cycle of order $m$ (in particular, $\overrightarrow{C_3}=G_1$).  If $m$ and $n$ are integers larger than $3$, it follows that
\[\gamma(\overrightarrow{C}_n \cp \overrightarrow{C}_m) \ge \left\lceil\frac{nm}{3}\right\rceil\ge \left\lceil\frac{n}{2}\right\rceil \left\lceil\frac{m}{2}\right\rceil
=\gamma(\overrightarrow{C}_n)\gamma(\overrightarrow{C}_m)\,.\]
That is, for two strongly connected cycles, both of order at least $4$, we see they satisfy the inequality in the digraph version of Vizing's conjecture.

%%%%%%%%%%%%%%%%%%%%%%%%%%%%%%%%%%%%%%%%
\section{Equality in Vizing's Inequality} \label{sec:equality}
%%%%%%%%%%%%%%%%%%%%%%%%%%%%%%%%%%%%%%%%
By Corollary~\ref{cor:allditrees}, every ditree satisfies Vizing's inequality, and it would be interesting to know for which ditrees $T$ and digraphs $H$ the equality $\gamma(T\cp H)=\gamma(T)\gamma(H)$ holds.
More generally, the same question can be posed for any digraph that satisfies Vizing's inequality. In the undirected case, the question of determining the pairs of graphs attaining the equality in Vizing's inequality has been studied by several authors, where we specifically mention the work of Fink, Jacobson, Kinch, and Roberts \cite{FJKR} as well as Jacobson and Kinch~\cite{JK}. It turns out that the graphs $C_4, P_4$ and corona graphs play an important role in many constructions of families of graphs with $\gamma(G\cp H)=\gamma(G)\gamma(H)$; see~\cite[Section 7.5]{hr-1998}.

Note that there exist four non-isomorphic orientations of the graph $C_4$, which can be distinguished by the sequence of the outdegrees of vertices presented in the circular order, and written as follows: $C_4^{(0,2,1,1)}$, $C_4^{(0,1,2,1)}$, $C_4^{(0,2,0,2)}$, and $C_4^{(1,1,1,1)}=\overrightarrow{C_4}$. Interestingly,
all four digraphs have packing number and domination number equal to $2$, and thus each satisfies Vizing's inequality as can be seen by applying Proposition~\ref{prp:packing}.

We take a closer look at $C_4^{(0,2,0,2)}$ depicted in Figure~\ref{fig:C_4}, with vertices $u$ and $v$ having outdegree $2$, and present a large family of digraphs $G$ with $\gamma(G\cp C_4^{(0,2,0,2)})=\gamma(G) \gamma(C_4^{(0,2,0,2)})$.

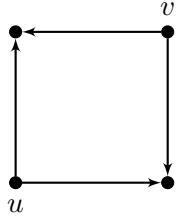
\begin{figure}[h!]
\begin{center}
\begin{tikzpicture}[auto]
\tikzset{edge/.style = {->,> = latex'}}
    % Place nodes

	\vertex (1) at (0,0) [scale=.75pt,fill=black,label=below:$u$]{};
	\vertex (2) at (0,2) [scale=.75pt,fill=black]{};
	\vertex (3) at (2,0) [scale=.75pt,fill=black]{};
	\vertex (4) at (2, 2) [scale=.75pt,fill=black,label=above:$v$]{};
	\draw[edge, thick] (1) to (2);
	\draw[edge, thick] (1) to (3);
	\draw[edge, thick] (4) to (2);
	\draw[edge, thick] (4) to (3);

\end{tikzpicture}
\end{center}
\caption{Particular directed version of $C_4$.}
\label{fig:C_4}
\end{figure}

\begin{lem} \label{lem:C4}
If $G$ is a digraph such that $V(G)$ can be partitioned into two dominating sets, then $\gamma(G\cp C_4^{(0,2,0,2)})\le n(G)$.
\end{lem}
\begin{proof}
Let $A$ and $B$ be dominating sets of $G$, where $A\cap B=\emptyset$ and $A\cup B=V(G)$. We claim that the set $S=\{(x,u):\ x\in A\}\cup\{(x,v):\ x\in B\}$ is a dominating set of $G\cp C_4$. Let $(x,h)\in V(G\cp C_4^{(0,2,0,2)})$. If $h\notin\{u,v\}$, then it is clear that either $(x,u)\in S$ or $(x,v)\in S$, and so $(x,h)$ is dominated by one of these two vertices. Suppose that  $(x,h)\in V(G\cp C_4^{(0,2,0,2)})$, and assume that $h=u$. Now, if $x\in A$, then $(x,u)\in S$, so $(x,u)$ is dominated by itself. On the other hand, if $x\in B$, then there exists $y\in A$ such that $yx$ is an arc in $G$ (since $A$ is a dominating set of $G$). Since $(y,u)(x,u)$ is an arc in $G\cp C_4^{(0,2,0,2)}$, we derive that $(x,u)$ is dominated by $(y,u)\in S$. In the same way we prove that $(x,v)$ is dominated by $S$.
Since $S$ is a dominating set of $G\cp C_4^{(0,2,0,2)}$ and $|S|=n(G)$, we derive that $\gamma(G\cp C_4^{(0,2,0,2)})\le n(G)$.
\end{proof}

By using Lemma~\ref{lem:C4}, we prove the following result for digraphs that satisfy Vizing's inequality.

\begin{prop}\label{prp:C4} If $G$ is a digraph such that $V(G)$ can be partitioned into two minimum dominating sets, then $\gamma(G\cp C_4^{(0,2,0,2)})=\gamma(G)\gamma(C_4^{(0,2,0,2)})$.
\end{prop}
\begin{proof} By the above observations, the following is straightforward:
$$2\gamma(G)=\gamma(G)\gamma(C_4^{(0,2,0,2)})\le \gamma(G\cp C_4^{(0,2,0,2)})\le |V(G)|= 2\gamma(G).$$
Note that the last inequality follows from Lemma~\ref{lem:C4} since $V(G)$ can be partitioned into two minimum dominating sets. Hence all expressions in the chain of inequalities are equal, and in particular, $\gamma(G\cp C_4^{(0,2,0,2)})=\gamma(G)\gamma(C_4^{(0,2,0,2)})$.
\end{proof}

Let $\mathcal{C}$ represent the class of all digraphs whose underlying graph is a corona of a tree. By Corollary~\ref{cor:allditrees}, since the digraphs in $\cal C$ are ditrees, we infer that they satisfy Vizing's inequality.
In addition, it is clear that any digraph $G\in {\cal C}$ with the property that each leaf $\ell$ has ${\rm indeg}(\ell)=1={\rm outdeg}(\ell)$ admits the conditions of Proposition~\ref{prp:C4}. Indeed, a partition, where one of the subsets contains exactly the leaves, does the trick.

The following example shows that both arcs on the corona edges for a graph in ${\cal C}$ is not a necessary condition to achieve equality in the Vizing's  inequality for digraphs.
Let $D$ be the digraph in Figure~\ref{fig:cor(P3)}.  Note that both of $ux$ and $xu$ are arcs in $D$; however, $vy$ and $wz$ are both
arcs in $D$ but neither of $yv$ or $zw$ is an arc in $D$.  It is clear that $D$ satisfies the conditions of Theorem~\ref{prp:C4}. Indeed, take the partition $A=\{u,w,y\},B=\{v,x,z\}$. Therefore, $\gamma(C_4^{(0,2,0,2)} \cp D)=2\cdot 3=\gamma(C_4^{(0,2,0,2)})\gamma(D)$.

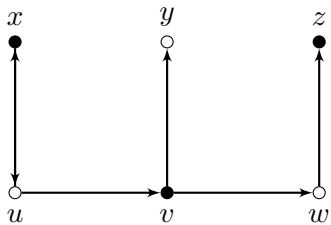
\begin{figure}[h!]
\begin{center}
\begin{tikzpicture}[auto]
\tikzset{edge/.style = {->,> = latex'}}
    % Place nodes
	\vertex (1) at (0,0) [scale=.75pt,fill=white,label=below:$u$]{};
	\vertex (2) at (2,0) [scale=.75pt,fill=black, label=below:$v$]{};
	\vertex (3) at (4,0) [scale=.75pt,fill=white, label=below:$w$]{};
    \vertex (4) at (0,2) [scale=.75pt,fill=black, label=above:$x$]{};
    \vertex (5) at (2,2) [scale=.75pt,fill=white, label=above:$y$]{};
    \vertex (6) at (4,2) [scale=.75pt,fill=black, label=above:$z$]{};
	\draw[edge, thick] (1) to (2);
	\draw[edge, thick] (2) to (3);
    \draw[edge, thick] (1) to (4);
    \draw[edge, thick] (4) to (1);
    \draw[edge, thick] (2) to (5);
    \draw[edge, thick] (3) to (6);
\end{tikzpicture}
\end{center}
\caption{A digraph $D \in \mathcal{C}$ with vertex shading indicating the partition.}
\label{fig:cor(P3)}
\end{figure}

We next focus on ditrees, which attain equality in Vizing's inequality when multiplied with an appropriate digraph. First, we present a necessary condition for these ditrees.

\begin{thm}
\label{thm:strongsupport} If $G$ is a digraph whose underlying graph is connected and $T$ is any ditree such that $\gamma(T\cp G) = \gamma(T)\gamma(G)$, then $T$ does not contain a strong support vertex adjacent to two non-isolated leaves.
\end{thm}
\begin{proof}
Let $S$ be a minimum dominating set of $T\cp G$, and for each vertex $x$ of $T$, let $S_x=S \cap V(^x\!G)$.  For each support vertex $v$ of $T$, let $I_v$ represent the set,
which might be empty, of isolated leaves adjacent to $v$. We claim there exists a maximum packing $P$ of $T$ such that for each support vertex $v$ of $T$, $I_v \subseteq P$ and exactly one non-isolated leaf of $v$ (should it exist) is in $P$. To see this, from among all maximum packings of $T$, choose  $P$ containing the maximum number of leaves.   Let $v$ be any support vertex of $T$. If $v \in P$, then $(P - \{v\}) \cup I_v \cup \{\ell\}$ where $\ell$ is any non-isolated leaf of $v$ (should it exist) is another maximum packing containing more leaves than $P$, which contradicts our choice of $P$. Therefore, $v \not\in P$. Note that this implies that $I_v \subseteq P$ for otherwise $P$ is not maximum. Moreover, if no support of $T$ is adjacent to a non-isolated leaf, then $P$ is our desired packing. Therefore, we shall assume that $v$ is adjacent to at least one non-isolated leaf. Suppose there exists some $w \in P$ which is a non-leaf out-neighbor of $v$. We claim that $(P- \{w\})\cup \{\ell\}$ where $\ell$ is a non-isolated leaf of $v$ is also a packing of $T$. Indeed, suppose $t \in P$ is a non-leaf in- or out-neighbor of $v$ other than $w$. Then $vt \not\in A(T)$ since $vw \in A(T)$ and $P$ is a packing. Thus, $\ell$ has no common in-neighbor with $t$ and $(P - \{w\})\cup \{\ell\}$ is in fact a packing containing more leaves than $P$, which is a contradiction. Therefore, $P$ does not contain a non-leaf out-neighbor of $v$. It follows that $P$ contains a non-isolated leaf of $v$, for otherwise $P$ is not maximum. Thus, $P$ is our desired packing. Enumerate the vertices of $P$ as $x_1, \dots ,x_k$. We construct a partition $\Pi= \Pi_1 \cup \dots \cup\Pi_k$ of $V(T)$ as follows. Let $x_i \in \Pi_i$. If $v$ is an in-neighbor of $x_i$, place $v \in \Pi_i$. For all remaining $v$, choose the smallest index $i$ such that either $v$ is an out-neighbor of $x_i$ or $v$ and $x_i$ have a common in-neighbor. Notice that all vertices of $T$ have been placed in some $\Pi_i$, for otherwise since $\ugr{T}$ is connected, there exists $x \in V(T)$ such that $x$ is neither an in-neighbor nor an out-neighbor of any vertex of $P$ and it does not share a common in-neighbor with a vertex of $P$. This implies that $P$ is not maximum.

As noted in the proof of Proposition~\ref{prp:packing}, $|S \cap (\Pi_i\times V(G))| \ge \gamma(G)$ for all $i \in [k]$. Furthermore, $\rho(T)=\gamma(T)$ by Theorem~\ref{thm:packing-dom-ditrees}, and hence $|S| = \gamma(T)\gamma(G)=k\gamma(G)$, which implies that $|S \cap (\Pi_i \times V(G))| = \gamma(G)$, for every $i \in [k]$. Let $v$ be a strong support vertex of $T$ adjacent to two non-isolated leaves $w_1$ and $w_2$. Whether or not $w_1$ or $w_2$ is in $P$, we know that $v, w_1$, and $w_2$ are in the same set $\Pi_i$ for some $1 \le i \le k$. Reindexing if necessary, we may assume $\{v, w_1, w_2\}\subseteq \Pi_1$. We also know that $|S_{w_1}| + |S_v| \ge \gamma(G)$ and $|S_{w_2}| + |S_v| \ge \gamma(G)$. Assuming that $|S_{w_1}| \le |S_{w_2}|$, we have
\[2|S_v| + 2|S_{w_2}| \ge 2|S_v| + |S_{w_1}| + |S_{w_2}| \ge 2 \gamma(G).\]
On the other hand, $|S_v| + |S_{w_2}| \le \gamma(G)$ since $|S \cap (\Pi_1\times V(G)| = \gamma(G)$. Therefore, $|S_{w_1}| = 0$, which implies that $S_v = V(G)$.  This is a contradiction since  $\gamma(G) \le n(G) -1$ follows from the fact that $\ugr{G}$ is connected. Thus, $T$ does not contain a strong support adjacent to two non-isolated leaves.
\end{proof}

The next result shows that adding isolated leaves arbitrarily can be used to enlarge the family of digraphs (in particular, ditrees) attaining the Vizing equality.

\begin{prop}
\label{prp:attachisolated}
Let $G$ and $H$ be digraphs such that $\gamma(G\cp H)=\gamma(G)\gamma(H)$. If $G'$ is obtained from $G$ by attaching a new isolated leaf to any vertex of $G$ such that $\gamma(G')=\gamma(G)+1$ and $G'$ satisfies Vizing's inequality, then $\gamma(G'\cp H)=\gamma(G')\gamma(H)$.
\end{prop}
\begin{proof}
Let $V(G')=V(G)\cup \{x\}$ and $A(G')=A(G)\cup\{xv\}$, where $v\in V(G)$ is a vertex of $G$ such that $\gamma(G')= \gamma(G)+1$ and $G'$ satisfies Vizing's inequality. Note that $\gamma(G'\cp H)\le \gamma(G)\gamma(H)+\gamma(H)$. Now, if $\gamma(G'\cp H)<\gamma(G)\gamma(H)+\gamma(H)=\gamma(G')\gamma(H)$, this is a contradiction with $G'$ being a digraph that satisfies Vizing's inequality. Therefore, $\gamma(G'\cp H)=\gamma(G)\gamma(H)+\gamma(H)=\gamma(G')\gamma(H)$.
\end{proof}

\begin{cor} \label{cor:attachisolated}
Let $T$ be a ditree and $H$ a digraph such that $\gamma(T\cp H)=\gamma(T)\gamma(H)$. If $T'$ is obtained from $T$ by attaching a new isolated leaf to any vertex of $T$ such that $\gamma(T')=\gamma(T)+1$, then $\gamma(T'\cp H)=\gamma(T')\gamma(H)$.
\end{cor}

Hence, in view of the above corollary, in searching for the ditrees that attain equality in Vizing's inequality, we can restrict to ditrees $T$ whose isolated vertices are essential for keeping the domination number of $T$ (that is, removing an isolated leaf from $T$ drops the domination number).

In the next result we present some properties that two ditrees which attain the equality in Vizing's inequality must enjoy.

\begin{thm} Let $T_1$ and $T_2$ be ditrees of order at least $3$ such that $\gamma(T_1 \cp T_2) = \gamma(T_1) \gamma(T_2)$. If $P_1$ and $P_2$ are maximum packings of $T_1$ and $T_2$, respectively, then $P_i$ dominates $\ugr{T_i}$ for each $i\in [2]$. Moreover, every maximum packing of $T_1$ or $T_2$ contains all isolated leaves.
\end{thm}
\begin{proof} First, notice that $P_1 \times P_2$ is a packing of $T_1 \cp T_2$. Suppose $P_1$ does not dominate $u$ in $\ugr{T_1}$ and pick any vertex $y \in V(T_2) - P_2$. We claim that $(P_1 \times P_2) \cup\{(u, y)\}$ is a packing of $T_1\cp T_2$. To see this, note that $(P_1 \times P_2) \cup \{(u, y)\}$ is an independent set. Suppose there exists $(x_1, x_2) \in P_1 \times P_2$ such that $(s, t)$ is an in-neighbor of both $(x_1, x_2)$ and $(u, y)$. It follows that either $s = x_1$ or $t = x_2$. Suppose first that $s = x_1$. Thus, $x_1u \in A(T_1)$ meaning that $P_1$ dominates $u$ in $\ugr{T_1}$, contradicting our assumption. Therefore, we shall assume that $t=x_2$. Hence, $s = u$ and $sx_1 \in A(T_1)$ implying that $P_1$ dominates $u$ in $\ugr{T_1}$, which is another contradiction.  Thus, $(x_1, x_2)$ and $(u, y)$ have no common in-neighbor and $(P_1 \times P_2) \cup \{(u, y)\}$ is a packing of $T_1 \cp T_2$.  Applying Theorem~\ref{thm:packing-dom-ditrees} we infer that $\gamma(T_1)\gamma(T_2)= \gamma(T_1\cp T_2) \ge \rho(T_1\cp T_2) > \rho(T_1)\rho(T_2) = \gamma(T_1)\gamma(T_2)$, which is not possible. Therefore, $P_1$ dominates $\ugr{T_1}$. A similar argument shows that $P_2$ dominates $\ugr{T_2}$.
Finally, suppose that there exists an isolated leaf $u_1 \not\in P_1$ and an isolated leaf $u_2 \not\in P_2$. Then $(P_1 \times P_2) \cup \{(u_1, u_2)\}$ is a packing in $T_1 \cp T_2$ for it is an independent set and $(u_1, u_2)$ has no in-neighbor. However, this implies that $\gamma(T_1 \cp T_2) > \gamma(T_1)\gamma(T_2)$. So either $P_1$ contains all isolated leaves of $T_1$ or $P_2$ contains all isolated leaves of $T_2$. \end{proof}

Let us now present a large infinite family of ditrees that attain the equality in Vizing's inequality. More precisely, if $T$ is an arbitrary ditree, we construct a ditree $T^*$ that contains $T$ as a subdigraph, and where $\gamma(T^*\cp P_4)=\gamma(T^*)\gamma(P_4)$. By $P_4$ we denote the digraph whose underlying graph is $P_4$ such that between every two adjacent vertices there are arcs in both directions, and let $V(P_4)=\{v_1,v_2,v_3,v_4\}$.

\begin{figure}[h!]
\begin{center}
\begin{tikzpicture}[auto]
\tikzset{edge/.style = {->,> = latex'}}
    % Place nodes

	\vertex (a) at (0,0) [scale=.75pt,fill=white, label=below:$a$]{};
	\vertex (b) at (1.5,0) [scale=.75pt,fill=white,label=below:$b$]{};
	\vertex (c) at (3,0) [scale=.75pt,fill=white, label=below:$c$]{};
	\vertex (x) at (4.5,0) [scale=.75pt,fill=white, label=below:$x$]{};
	\vertex (c') at (6,0) [scale=.75pt,fill=white, label=below:$c'$]{};
	\vertex (b') at (7.5,0) [scale=.75pt,fill=white, label=below:$b'$]{};
	\vertex (a') at (9,0) [scale=.75pt,fill=white, label=below:$a'$]{};
	
	\vertex (1) at (-1.5,1.5) [scale=.75pt,fill=white, label=left:$v_1$]{};
    \vertex (2) at (-1.5,3) [scale=.75pt,fill=white, label=left:$v_2$]{};
    \vertex (3) at (-1.5,4.5) [scale=.75pt,fill=white, label=left:$v_3$]{};
   	\vertex (4) at (-1.5,6) [scale=.75pt,fill=white, label=left:$v_4$]{};
   	 \foreach \i in {1,...,4}
   	{
   	\vertex (a\i) at (0,1.5*\i) [scale=.75pt,fill=white]{};
	\vertex (b\i) at (1.5,1.5*\i) [scale=.75pt,fill=white]{};
	\vertex (c\i) at (3,1.5*\i) [scale=.75pt,fill=white]{};
	\vertex (x\i) at (4.5,1.5*\i) [scale=.75pt,fill=white]{};
	\vertex (c'\i) at (6,1.5*\i) [scale=.75pt,fill=white]{};
	\vertex (b'\i) at (7.5,1.5*\i) [scale=.75pt,fill=white]{};
	\vertex (a'\i) at (9,1.5*\i) [scale=.75pt,fill=white]{};
	}
	
	\vertex (a2) at (0,3) [scale=.75pt,fill=black]{};
	\vertex (a3) at (0,4.5) [scale=.75pt,fill=black]{};
	\vertex (c1) at (3,1.5) [scale=.75pt,fill=black]{};
	\vertex (c4) at (3,6) [scale=.75pt,fill=black]{};
	\vertex (c'2) at (6,3) [scale=.75pt,fill=black]{};
	\vertex (c'3) at (6,4.5) [scale=.75pt,fill=black]{};
	\vertex (a'1) at (9,1.5) [scale=.75pt,fill=black]{};
	\vertex (a'4) at (9,6) [scale=.75pt,fill=black]{};

	\draw[edge, thick] (a) to (b);
	\draw[edge, thick] (a') to (b');
	\draw[edge, thick] (c') to (x);
	\draw[edge, thick] (c) to (x);
	\draw[edge, thick] (a) to (b);
	\draw[edge, thick] (b) to (c);
	\draw[edge, thick] (c) to (b);
	\draw[edge, thick] (b') to (c');
	\draw[edge, thick] (c') to (b');
	
	\draw[edge, thick] (1) to (2);
	\draw[edge, thick] (2) to (1);
	\draw[edge, thick] (3) to (2);
	\draw[edge, thick] (2) to (3);
	\draw[edge, thick] (3) to (4);
	\draw[edge, thick] (4) to (3);
	
\end{tikzpicture}
\end{center}
\caption{The Cartesian product $K_1^*\cp P_4$ with vertices of a dominating set darkened.}
\label{fig:family}
\end{figure}
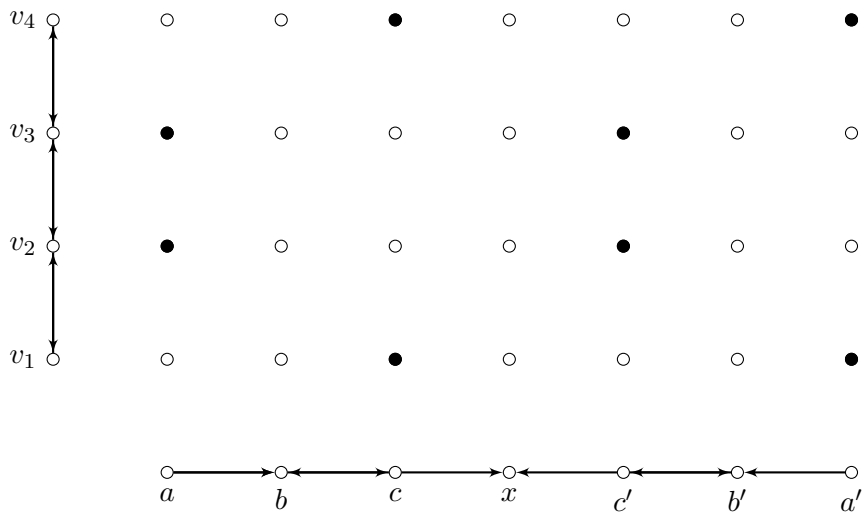

The ditree $K_1^*$ is depicted at the bottom of Figure~\ref{fig:family}, and $P_4$ is on the left in that figure. Clearly, $\gamma(K_1^*)=4$ and $\gamma(P_4)=2$. The figure represents the Cartesian product $K_1^*\cp P_4$, where arcs are omitted for a clearer presentation, and dark vertices present a (minimum) dominating set of $K_1^*\cp P_4$.
We claim that $\gamma(K_1^*\cp P_4)=8=\gamma(K_1^*)\gamma(P_4)$.
Indeed, since $$S=\{(a,v_2),(a,v_3),(c,v_1),(c,v_4),(c',v_2),(c',v_3),(a',v_1),(a',v_4)\}$$ is a dominating set of $K_1^*\cp P_4$, we infer $\gamma(K_1^*\cp P_4)\le 8$, while the reversed inequality follows by Corollary~\ref{cor:allditrees}. More generally, let $T^*$ be obtained from a ditree $T$, by taking $n(T)$ copies of the ditree $K_1^*$, and identify each vertex $v$ of $T$ with the vertex $x$ of its own copy of $K_1^*$. Note that $\gamma(T^*)=4n(T)$, and $\gamma(T^*\cp P_4)=8n(T)=2\gamma(T^*)$. In particular, this example shows that the equality in Vizing's inequality can be attained by ditrees with arbitrarily large vertex in- or out-degrees.

The problem of finding the pairs of ditrees that attain the equality in Corollary~\ref{cor:allditrees} seems to be difficult. The special version of this problem for  undirected trees was posed in~\cite{JK}, and is to our knowledge still unresolved.

\begin{prob}
Characterize the pairs of ditrees $T_1$ and $T_2$ such that $\gamma(T_1\cp T_2)=\gamma(T_1)\gamma(T_2)$.
\end{prob}

The following question is also interesting.

\begin{prob}
For which ditrees $T$ does there exist a digraph $G$ such that $\gamma(T\cp G)=\gamma(T)\gamma(G)$?
\end{prob}

As we know from Theorem~\ref{thm:strongsupport}, $\gamma(T\cp G)=\gamma(G)\gamma(T)$ implies that every strong support vertex of $T$ is adjacent to at most one non-isolated leaf.  In addition, Corollary~\ref{cor:attachisolated} may be used in obtaining large families of ditrees from a given ditree $T$ enjoying $\gamma(T\cp G)=\gamma(G)\gamma(T)$. Indeed, if we attach any number of isolated leaves to a support vertex of $T$ and call the resulting tree $T'$, then $\gamma(T'\cp G)=\gamma(T')\gamma(G)$.

%%%%%%%%%%%%%%%%%%%%%%%%%%%%%
%%%%%%%%%%%%%%%%%%%%%%%%%%%%%
\section*{Acknowledgments}

The first author was supported by the Slovenian Research Agency (ARRS) under the grants P1-0297, J1-9109 and J1-1693.

\end{document}